\documentclass{elsarticle}

\usepackage{xcolor}
\usepackage{amsmath}
\usepackage{algorithm}
\usepackage{algpseudocode}
\usepackage{float}
\usepackage{multirow}

\setcitestyle{square}

\makeatletter
\def\BState{\State\hskip-\ALG@thistlm}
\makeatother
\usepackage{graphicx}
\usepackage[margin=0.95in]{geometry}
\usepackage{amsmath,amsthm,amsfonts,amssymb,amstext,euscript,verbatim,bbm}

\theoremstyle{definition}

\newtheorem{theorem}{Theorem}[section]
\newtheorem{definition}{Definition}[section]

\newtheorem{remark}{Remark}[section]

\numberwithin{equation}{section}

\DeclareMathOperator*{\esssup}{esssup}

\journal{Journal of \LaTeX\ Templates}

\begin{document}
\begin{frontmatter}
\title{Deep Neural Network Approach to Forward-Inverse Problems}

\author[mymainaddress]{Hyeontae Jo\corref{cor1}\fnref{fn1}}
\ead{jht0116@postech.ac.kr}
\author[mymainaddress]{Hwijae Son\corref{cor1}\fnref{fn1}}
\ead{son9409@postech.ac.kr}
\fntext[fn1]{Both authors contributed equally to this work}

\author[mymainaddress]{Hyung Ju Hwang\corref{mycorrespondingauthor}}
\ead{hjhwang@postech.ac.kr}
\cortext[mycorrespondingauthor]{Corresponding author}

\author[mysecondaryaddress]{Eunheui Kim}
\ead{EunHeui.Kim@csulb.edu}

\address[mymainaddress]{Department of Mathematics, Pohang University of Science and Technology, South Korea}
\address[mysecondaryaddress]{Department of Mathematics and Statistics, California State University Long Beach, US}




\begin{abstract}In this paper, we construct approximated solutions of Differential Equations (DEs) using the Deep Neural Network (DNN). Furthermore, we present an architecture that includes the process of finding model parameters through experimental data, the inverse problem. That is, we provide a unified framework of DNN architecture that approximates an analytic solution and its model parameters simultaneously.  The architecture consists of a feed forward DNN with non-linear activation functions depending on DEs, automatic differentiation \cite{baydin2018automatic}, reduction of order, and gradient based optimization method. We also prove theoretically that the proposed DNN solution converges to an analytic solution in a suitable function space for fundamental DEs. Finally, we perform numerical experiments to validate the robustness of our simplistic DNN architecture for 1D transport equation, 2D heat equation, 2D wave equation, and the Lotka-Volterra system.
\end{abstract}
\allowdisplaybreaks

\end{frontmatter}

\section{Introduction}
\label{sec1}
This paper marks the first step toward a comprehensive study on deep learning architectures to solve forward-inverse problems for differential equations. Recent advances in deep learning show its capability to handle various types of model problems in many disciplines. In particular, deep learning techniques have been applied to understand data augmented differential equations. While most of such studies have been centered around heuristics and modeling prospectives, to the best of our knowledge, there is little to no theoretical analysis to confirm whether the deep learning architectures give rise to the correct solutions to the governing differential equations. An overreaching goal of this paper is to provide a comprehensive analysis of Deep Neural Networks (DNNs) to solve data-driven differential equations. This paper reports a novel architecture leveraging recent progress in deep learning techniques that solves forward-inverse problems for differential equations. The paper further includes a convergence analysis and its experimental results for our deep learning architecture. 

Forward-inverse problems (or inverse problems in short) for differential equations in this paper are related to data augmented differential equations. Namely, we consider equations of states for physical systems as governing differential equations and model parameters such as advection rates, reaction diffusion coefficients, for example, that need to be fitted by the given data. Hence numerical methods solving forward-inverse problems typically become constraint problems that require an ensemble of two steps, (1) solve the state equations, which is called the forward problems and (2) find the correct model parameters that fit the given data set. Inverse problem is an actively studied field and many numerical algorithms for the inverse problems are  robust enough to handle sufficiently large data sets, see for example \cite{arloff2018parameter, chavent2010nonlinear, levenberg1944method,  li2018particle, marquardt1963algorithm, tsilifis2016computationally, yaman2013survey} references therein.  However, such algorithms can be computationally expensive and they may be too sophisticated for non-experts in inverse problems to implement them. This calls for simplistic methods that unify two steps in solving forward-inverse problems. 

The contributions of this paper are three-fold. First, the DNN architecture presented in this paper highlights its simplistic approach to handling forward-inverse problems simultaneously. Second, a rigorous analysis of the convergence of the DNN solutions to the actual solutions for the governing problems is provided. Third, numerical experiments validate the robustness of our simple architecture. 

The paper comprises the followings. A short overview of related works on data-driven differential equations and the problem formulation are presented in the rest of Section \ref{sec1}. The methodology including the DNN architecture and loss function is described in Section \ref{sec2}. Theoretical results are given in Section \ref{sec3}. Section \ref{sec4} is devoted to the experiments done for the problem. In Section \ref{sec5}, we conclude the paper. 

\subsection{Background}
\label{Background_sec}
There are many works to utilize an Artificial Neural Network (ANN) to solve Differential Equations (DEs) in place of using well established numerical methods such as finite difference, finite element, and finite volume methods. Those finite schemes are heavily depending on mesh-grid points, and they may become a hindrance when the state equations reside in a domain with complex geometry. As such, a mesh-free approximation using clever constructions of basis functions has been introduced, see for example \cite{fasshauer1996solving, sarra2005adaptive}) and references therein. The concept of using ANN to solve DEs can be related to mesh-free approximations as one can consider Multi-Layer Perceptrons (MLPs) as an approximation to solutions to DEs. 

The ANN structure was first introduced in \cite{mcculloch1943logical}. Several studies follow to identify a class of functions that can be recognized by ANNs. The following is a brief overview and a few highlights of such studies. Cybenko \cite{cybenko1989approximation} established sufficient conditions for which a continuous function can be approximated by finite linear combinations of single hidden layer neural networks with the same univariate function. 
About the same time, Hornik {\em et al} \cite{hornik1989multilayer} showed that measurable functions can be approximated by multi-layer feedforward networks with a monotone sigmoid function as the authors called them "universal approximators". Cotter \cite{cotter1990stone} extends the result of \cite{hornik1989multilayer} to new architectures, and later Li \cite{li1996simultaneous} showed that a MLP with one hidden layer can approximate a function and its higher partial derivatives on a compact set.

The concept of using ANN to solve DEs are not new, and it has gained much attention recently.
Some of the highlights include the following. Lagaris {\em et al} \cite{lagaris1998artificial} studied to solve DEs and PDEs using an ANN with architecture including 1 single layer and 10 units and they next extended in \cite{lagaris2000neural} their results to a domain with complex geometry. Jianyu {\em et al} \cite{jianyu2003numerical} used ANN with a radial basis function as an activation function for Poisson equations. More recently, Berg {\em et al} in \cite{berg2018unified} used a DNN to solve steady (time-independent) problems for a domain with complex geometry in one and two space dimensions, and later in  \cite{berg2017neural} they studied DNN architectures to solve augmented Poisson equations (inverse problems) including three space dimensions. 

The recent work by Raissi {\em et al} \cite{raissi2019physics} can be perhaps closely related to our work in the sense that their DNN architectures (they called "continuous time models") resemble ours. We note however the aim of this paper is to establish a comprehensive study that includes a rigorous convergence analysis of DNNs to solve forward-inverse problems for DEs and PDEs. Since our convergence result is for linear equations at this point, we present our experiments for the well known linear equations as well. While our experiments cover simpler equations than those studied in \cite{raissi2019physics}, as their focuses were on architectures for "data-efficient and physics informed learning machines", we hope that our result can enhance the experiments shown in \cite{raissi2019physics}. We believe our first comprehensive results can shed lights onto further studies toward more complex and complicated systems using machine learning architectures.

\subsection{Problem formulation}
\label{Problem_sec}
We consider the equations of states as the following time dependent initial boundary value problems:
\begin{align}
L_{p}u &= 0,\text{  }t\in (0,T],\text{  }x\in \Omega,\label{ge}\\
Iu &= f,\text{  }x\in \Omega,\label{ic}\\
Bu &= g,\text{  }t\in (0,T],\text{  }x\in\partial\Omega, \label{bc}
\end{align}
where $L_{p}$ is a differential operator, $p$ is a set of model parameters, $\Omega\subset\mathbb{R}^d$ is a bounded domain (for the position of the state $u$), $\partial\Omega$ is a boundary of $\Omega$, $f(x)$ is an initial distribution, $I$ is an initial operator, $B$ is a boundary operator, and $g(t,x)$ is a boundary data. 
Next the governing equation is equipped with a set of observation data, which may be provided from actual experiments, as following:
\begin{equation}\label{obs_data}
D=\{(t_{i},x_{j},u_{ij})|i=1,2,...,n,\text{  }j=1,2,...,m\},
\end{equation}
where $u_{ij}$ denotes the output value (observation from the experiment) at position $x_{j}\in\Omega$ and time $0<t_i\leq T$ with the final time $T$. $n$ and $m$ refer the numbers of time and spatial observations respectively. Since our experiments cover ODEs and PDEs with one and two spatial dimensions, $x_j$ dependencies and the index $j$ for $x_j$ will be adjusted for each example.

We apply DNNs to solve governing equations \eqref{ge}-\eqref{bc} and the observation data \eqref{obs_data}. Our loss function includes the governing equation, observations, initial and boundary conditions. The loss function is minimized by using the standard gradient descent method. The results show that our DNN architecture can handle much lesser numbers of observations compared to known numerical solvers such as finite difference, element and volume methods. 
The results presented in this paper demonstrate that our DNN architecture is perhaps the most simplistic way to solve forward-inverse problems, yet robust enough to handle many different cases. Furthermore, we establish the convergence result of forward-inverse problems for parabolic or hyperbolic linear PDEs with Dirichlet boundary conditions. Specifically we show that the sequence of DNN solutions converges to an analytic solution in $L^{\infty}(0,T;H_0^{1}(\Omega))$ sense.

\section{Methodology}\label{sec2}
This section provides our DNN architecture and the mathematical formulation of the training procedure.

\subsection{DNN architecture}
 The DNN architecture can be written as a repeated composition of some affine transformations and some nonlinear activation functions. We denote by $u_N$ the DNN solution and assume that the architecture consists of $L + 1$ layers. The first layer takes $(t, x)$ as an input, and the last layer gives the output value $u_N(t,x)$. The $L-1$ layers between the first and the last layers are called hidden layers. We have used common nonlinear activation functions such as sigmoid, rectified linear units, and hyperbolic tangents through the DNN. Each neuron in the DNN contains a bias except for the input neurons. In two consecutive layers, the connections between neurons are written as weight matrices. Relations between $(l-1)^{th}$ and $l^{th}$ layers are defined by :

\begin{equation}\label{nn}
z^{l}_{j} = \sum_{k}^{N_{l-1}} w^{l}_{jk}\sigma_{l-1}(z^{l-1}_{k})+b^{l}_{j},
\end{equation}
where

\begin{itemize}
    \item $z^{l-1}_{k}$: $k^{th}$ neuron in $(l-1)^{th}$ layer
    \item $N_{l-1}$: the number of neurons in $(l-1)^{th}$ layer
	\item $b^{l}_{j}$: the bias of $j^{th}$ neuron in $l^{th}$ layer
	\item  $w^{l}_{jk}$: the weights between $k^{th}$ neuron in $(l-1)^{th}$ layer and $j^{th}$ neuron in $l^{th}$ layer
	\item $\sigma_{l-1}$ the activation function in $(l-1)^{th}$ layer
\end{itemize}
For convenience, we denote $z^{0}=(z^{0}_{1},...,z^{0}_{N_{0}})$ as $(t,x)$ and $z^{L}=(z^{L}_{1},...,z^{L}_{N_{L}})$ as $u_N(t,x)$, respectively. The values $N_{l-1}$, $L$, and the form of $\sigma_{l-1}$ should be chosen before training. In the training procedure, we have to calculate the optimal weights and biases $w^{l}_{jk}$, $b^{l}_{j}$ which minimize a suitable loss function. 

\subsection{Loss function}
 The training procedure of the DNN is equivalent to the optimization problem of the loss function with respect to the DNN parameters. We denote the DNN solution by $u_{N}(t,x)=u_{N}(t,x;w,b)$, where $(w,b)$ are the set of weights and biases defined in (\ref{nn}). Denote the number of grid points of time, spatial variables, initial and boundary domains by $N_t, N_x, I, B_t, B_x$ respectively. Now we define the loss function using (\ref{ge}),

\begin{equation}\label{loss_ge}
Loss_{GE}(w,b,p)=\int_{[0,T]} \int_{\Omega}(L_{p}u_{N}(t,x;w,b))^{2} dxdt \approx  \sum_{i,j=1}^{N_t,N_x}( L_{p}u_{N}(t_i,x_j;w,b))^{2},
\end{equation}
where the last approximated sum is obtained by sampling a grid point $\{(t_i,x_j)|t_i\in[0,T],x_j\in\Omega,\text{ for }i=1,...,N_t, j=1,...,N_x\}$. Note that the reason that we define the above loss function is to find the optimal weights $(w,b)$ which minimize $Loss_{GE}$. However, (\ref{loss_ge}) is not sufficient because it excludes information about initial and boundary conditions. Therefore we define two loss functions from (\ref{ic})-(\ref{bc})

\begin{align}
Loss_{IC}(w,b) &= \int_{\Omega}(Iu_{N}(0,x;w,b)-f(x))^{2}dx \approx\sum_{i=1}^{I}(u_{N}(0,x_i;w,b)-f(x_i))^2,\label{loss_ic}\\
Loss_{BC}(w,b) &= \int_{[0,T]} \int_{\partial\Omega}(Bu_{N}(t,x;w,b)-g(t,x))^{2}dSdt \approx\sum_{i,j=1}^{B_t, B_x}(u_{N}(t_i,x_j;w,b)-g(t_i,x_j)|)^{2}.\label{loss_bc}
\end{align}

Combining all loss functions (\ref{loss_ge})-(\ref{loss_bc}) is still not enough because the solution of DEs (\ref{ge})-(\ref{bc}) could differ depending on the choice of the equation parameter $p$. Due to this reason, we should make one additional loss function to calibrate $p$ using the observed data (\ref{obs_data}). 

\begin{equation}\label{loss_obs}
Loss_{Obs}(w,b) = \sum_{(t_i,x_j)\in D}|u_{ij}-u_{N}(t_i,x_j;w,b)|^2.
\end{equation}

Finally, we define the forward loss and total loss as a summation of all three, and four loss functions respectively. For the forward loss, the model parameter $p$ is considered to be fixed.
\begin{align}
&Loss_{Forward}(w,b) = Loss_{GE}(w,b) + Loss_{IC}(w,b) + Loss_{BC}(w,b), \label{loss_forward}\\
&Loss_{Total}(w,b,p) = Loss_{GE}(w,b,p) + Loss_{IC}(w,b) + Loss_{BC}(w,b) + Loss_{Obs}(w,b). \label{loss_total}
\end{align}


\begin{algorithm}
\caption{Training}\label{euclid}
\begin{algorithmic}[1]
\Procedure{train}{number of epochs}
    \For{number of epochs}
        \State sample minibatch of m samples $z^1, z^2,..., z^m$ from uniform distribution $p_{\Omega}(z)$
        \State sample minibatch of m samples $z_I^1, z_I^2,..., z_I^m$ from uniform distribution $p_{\{0\} \times\Omega }(z)$
        \State sample minibatch of m samples $z_B^1, z_B^2,..., z_B^m$ from uniform distribution $p_{\partial\Omega}(z)$
        \State sample k observation points $z_O^1, z_O^2,..., z_O^k$
        \State Find the true value $u_j = u_p(z_O^j)$ for $j=1,2,...,k$
        \State Update the neural network by descending its stochastic gradient :
        \begin{equation}\nonumber
            \nabla_{w, b} [\frac{1}{m} \sum_{i=1}^m [L_p(u_N)(z^i)^2 + (u_N(z_I^i)-f(z_I^i))^2 + (u_N(z_B^i)-g(z_B^i))^2] + \frac{1}{k}\sum_{j=1}^k (u_N(z_O^j)-u_j)^2]
        \end{equation}
    \EndFor
\EndProcedure
\end{algorithmic}
\end{algorithm}

\section{Theoretical result}\label{sec3}
This section provides a theoretical proof that there exists a sequence of weights such that the corresponding sequence of DNN solutions converges to an analytic solution on any compact subset of the domain. We focus on the DEs (\ref{ge})-(\ref{bc}) where the existence and the uniqueness of solutions are guaranteed. We establish the result in two steps. We first show that a sequence of DNN solutions converges to an analytic solution for the corresponding model parameters, called the forward problem. We next show that both the estimated parameter and the DNN solutions converge to the model parameter and the analytic solution simultaneously, called the inverse problem.

\subsection{Forward problem}
For the forward problem, we fix the model parameter $p$ and denote the analytic solution to (\ref{ge})-(\ref{bc}) by $u_p$. We also denote the DNN solution in (\ref{nn}) by $u_N$. In $u_N$, activation functions $\sigma$ are any non-polynomial functions in $C^{k}(\mathbb{R}^n)$.

Next we quote the following Definition \ref{C_hat} and Theorem \ref{global} from \cite{li1996simultaneous}

\begin{definition}\label{C_hat} Let $I_n=[0, 1]^n$ be the unit interval on $\mathbb{R}^n$, $k$ be a non-negative integer. Then we say a function $f$ is contained in $\widehat{C}^k(I_n)$ if $f\in C^k(\Omega)$ for some open set $U$ containing $I_n$
\end{definition}

\begin{theorem}\label{global}(Li, Theorem 2.1 in \cite{li1996simultaneous}) Let $f\in\widehat{C}^{k}(I_n)$. Then, given $\varepsilon>0$, we can find parameters of a neural network $u_N$, defined in (\ref{nn}) with $L=1$, so that
$$\|D^{\underline{k}}f-D^{\underline{k}}u_N\|_{L^{\infty}(I_n)}<\varepsilon$$
holds for any multi-index $\underline{k}=(k_1, k_2, ..., k_n)$, $|k_1|+|k_2|+...+|k_n|\leq k$, and $k_i$'s are non-negative integers.
\end{theorem}

\begin{remark} Since the above result can be generalized to multi-layer architectures (for example, \cite{hornik1989multilayer}) and to a general compact set $K$ instead of $I_n$, we may assume that the architecture contains only one hidden layer ($L=1$) and the domain $\Omega$ is $I_n$ ($\Omega = I_n$).
\end{remark}

\begin{theorem}\label{forward_loss} For a non-negative integer $k$, assume that the highest order of linear operator (\ref{ge}) is $k$ and $u_p\in\widehat{C}^k(I_n)$, and the activation function $\sigma(x)$ and its ($k$-th order) derivatives are continuous and discriminatory. Then, there exists $\{m_j, w_{j}, b_j\}_{j=1}^\infty$ such that a sequence of the DNN solutions with $m_j$ nodes, denoted by $\{u_j(w_j, b_j)\}_{j=1}^{\infty}$ satisties

\begin{equation}\label{forward_loss_0}
Loss_{Forward}(w_j,b_j)\rightarrow 0\text{ as }j\rightarrow\infty
\end{equation}

\end{theorem}

\begin{proof}
Let $\epsilon > 0$ be given. By Theorem \ref{global}, there exists a neural network $u_{j}(x) = \sum_{i=1}^{m_j} w_i^1 \sigma({w_i^2}x + b_i)$ such that $\|D^{\underline{k}}u_p - D^{\underline{k}}u_{j}\|_{\infty} < \epsilon$, where ${\underline{k}}$ is a non-negative multi-index up to differentiability of $u_p$. By integrating $|L_pu_{j}|^2\text{, } |Iu_{j}-f|^2\text{, } |Bu_{j} - g|^2 $ over $[0,T] \times \Omega\text{, } \Omega\text{, } [0,T] \times \partial\Omega$ respectively, we obtain the desired result.
\end{proof}

\begin{remark} The assumption $u_p\in\widehat{C}^k(I_n)$ in Theorem \ref{forward_loss} is a strong condition. Since we can also approximate it by a sequence of compactly supported smooth functions in $I_n$, we can extend the assumption to a general Sobolev space.
\end{remark}

The Theorem \ref{forward_loss} states that we can always find parameters of a DNN architecture which can reduce $Loss_{Forward}$ if the DE has a smooth analytic solution. However, since we can not directly use information of an analytic solution, we next show that the DNN architecture equipped with parameters which minimize $Loss_{Forward}$ converges to an analytic solution in Theorem \ref{theorem_forward}.

\begin{theorem}\label{theorem_forward} Let $L_{p}=\partial_{t}+L$ in (\ref{ge}) be a second order parabolic operator and $Bu=0$ in (\ref{bc}) be a Dirichlet boundary condition. Also we define the DNN solution $u_{j}=u_{j}(t,x;w_j,b_j)$ with $m_j$ nodes and the corresponding loss $Loss_{Forward}(w,b)$. Then,  $Loss_{Forward}(w,b)\rightarrow 0$ implies
\begin{equation}
u_{j}(t,x;w_j,b_j) \rightarrow u_p \text{ in } L^{\infty}([0,T];H^1_0(\Omega)),
\end{equation}
where $u_p$ is a solution to  (\ref{ge})-(\ref{bc})
\end{theorem}
\begin{proof}
First we assume that the activation function satisfies the Dirichlet boundary condition by replacing it with $b(x)\sigma(t,x)$, where $b(x)$ is an arbitrary smooth function that satisfies $b=0$ on $\partial \Omega$. By evaluating $u_p-u_{j}$ in (\ref{ge})-(\ref{bc}), we have the following
\begin{align}
\partial_{t}(u_p-u_{j})+L(u_p-u_{j}) &= \varepsilon_{m}(t,x),\text{  }t\in (0,T],\text{  }x\in \Omega,\nonumber\\
I(u_p-u_{j}) &= \eta_{m}(x),\text{  }x\in \Omega,\nonumber\\
B(u_p-u_{j}) &= 0,\text{  }t\in (0,T],\text{  }x\in\partial\Omega\nonumber
\end{align}
Then the energy estimates for the second order parabolic equation (see Theorem 5, Chapter 7 in  \cite{evans10}) are applied to obtain that
\begin{align}
\esssup_{0\leq t\leq T}\|u_p-u_{j}(\cdot,t)\|_{H^{1}_{0}(\Omega)}+\|u_p-u_{j}\|_{L^2([0,T];H^1_{0}(\Omega))}+\|\partial_{t}\left(u_p-u_{j}\right)\|_{L^2([0,T];L^2(\Omega))}\label{energy_parabolic}\\ \leq C\left(\|\varepsilon_{m}\|_{L^2([0,T];L^2(\Omega))}+\|\eta_{m}\|_{H^1_0(\Omega)}\right),\nonumber
\end{align}
where the constant $C$ in (\ref{energy_parabolic}) depends only on $\Omega, T$ and the coefficients of $L$. Note that the right hand side in (\ref{energy_parabolic}) is equivalent to $Loss_{Forward}(w,b)$. This shows that the sequence of DNN solutions converges to the analytic solution when  $Loss_{Forward}(w,b)\rightarrow 0$.
\end{proof}
\begin{remark}
The convergence result also holds when $L_{p}=\partial_{tt}+L$ is a second order hyperbolic operator.
\end{remark}
\subsection{Inverse problem}

\begin{definition}
Let $P$ be the set of all possible model parameters. Define $S := \{u_p\text{ }|\text{ }p \in P\}$ be the set of solutions corresponding to each model parameter $p \in P$.
\end{definition}

\begin{definition}
We say the observation set $D_p$ is $clear$ if for any $u_p, u_q \in S$, $u_p|_D = u_q|_D $ if and only if $p=q$
\end{definition}

\begin{theorem} Let $Loss_{total}(w,b,p)$ be the total loss defined in (\ref{loss_total}) and let the observation set $D_{p}$ with $p \in P$ be given and clear. We assume that for given $\{\epsilon_j\}_{j=1}^{\infty}$ with $\epsilon_j \rightarrow 0$, there exists $(m_j, w_j, b_j, p_j)$ such that $Loss_{Total}(w_j,b_j,p_j) < \epsilon_j$, and the parameter set $\{p_j\}$ is contained in  $P$, then 
\begin{equation}
u_{j}(w_j, b_j) \rightarrow u_{p} \text{ a.e. and } p_j \rightarrow p \text{ as } j \rightarrow \infty
\end{equation}
\end{theorem}
\begin{proof} We first divide the total loss into 
$Loss_{Total}=Loss_{Forward, p_{j}} + Loss_{Obs, p_{j}}$. 
For $m$ fixed, we set  $u_{N_{m,k}}(w_{m,k},b_{m,k})$ as a DNN solution with $k\geq m$ nodes, where $(w_{m,k}, b_{m,k})$ is a minimizer of $Loss_{Forward}$ defined in Definition \ref{forward_loss_0}. Then,  Theorem \ref{theorem_forward} implies $u_{N_{m,k}}\rightarrow u_{p_m}$ as $k\rightarrow\infty$. Also, $Loss_{Obs, p_{m}}\rightarrow 0$ implies $p_m \rightarrow p $ by definition of $D_p$

\end{proof}

\section{Experiments}\label{sec4}
In this section, we provide experimental results based on several differential equations including 1D transport equation, 2D heat equation, 2D wave equation, and the Lotka-Volterra system. For each equation, we have calculated an analytic (if possible) or numerical solution with fixed model parameters in order to generate a small amount of true solution points which will be regarded as the observation points. We apply our DNN model to find an approximated solution and the optimal equation parameter at the same time. In this step, we have used a neural network with variable depth and width, and ReLU activations. We used the Adam optimizer \cite{kingma2014adam, reddi2019convergence} with $(\beta1, \beta2) = (0.9, 0.999)$ in order to find the minimizer $w, b,$ and $p$ defined in (\ref{loss_total}). Also, for higher order derivatives in (\ref{ge}) we have applied the reduction of order technique to express it as a system of differential equations. This step dramatically reduces the computational cost for calculating (\ref{loss_ge}). For example, the second-order PDE $u_{xx}=f$ can be replaced by $v_x=f$ together with the equation $v=u_x$. That is, we derive two first-order PDEs $v_x=f, u_x=v$ from one second-order PDE $u_{xx}=f$, then the output layer of DNN should be changed into $(u,v)$. This method is applied to 2D Heat and Wave equations. Also, different types of activation functions are used depending on the behavioral characteristics of the governing equations. Finally, we provide two differences between 1) actual and model output values and 2) actual and calculated model parameters. Observation points were calculated from analytic (transport), series (heat, wave), numerical (Lotka-Volterra) solution and sampled randomly among them.

In the rest of this section, we present the experimental results. For each figure, top left, top right figures show our neural network solution and the analytic solution respectively. Bottom left shows the absolute error between the neural network solution and the analytic solution. Bottom right figure shows the convergence of estimated parameters to the real parameters. We have implemented our method by using Pytorch \cite{paszke2017automatic}, which is one of the most famous machine learning library. We first present the detailed experimental settings. Table \ref{grid_info} and \ref{nn_architecture} show the summarized information of the number of grid points and DNN architectures respectively.

\begin{table}
\centering
\caption {Information of grid and observation points}
\begin{tabular}{|c|c|c|c|}\hline
\multirow{2}{*}{ } & \multicolumn{3}{c|}{Data Generation}                        
\\ \cline{2-4}  & Grid Range & Number of Grid Points & Number of Observations \\ \hline
1D Transport      &     $(t,x) \in [0,1]\times[0,1]$       &   $17 \times 100$                  &        17            \\ \hline
2D Heat           &    $(t,x,y) \in [0,1]\times[0,1]\times[0,1]$        &         $100 \times 100 \times 100$              &       13                 \\ \hline
2D Wave           &     $(t,x,y) \in [0,1]\times[0,1]\times[0,1]$       &           $100 \times 100 \times 100$           &       61                 \\ \hline
Lotka-Volterra    &    $t \in [0,100]$        &            20,000           &       40                 \\ \hline
\end{tabular}
\label{grid_info}
\end{table}

\begin{table}
\centering
\caption {Neural network architecture}
\begin{tabular}{|c|c|c|c|}
\hline
\multirow{2}{*}{ } & \multicolumn{3}{c|}{Neural Network Architecture}             
\\ \cline{2-4} & Fully Connected Layers & Activation Functions & Learning Rate \\ \hline
1D Transport      & 2(input)-128-256-128-1(output)      &             ReLU         &      $10^{-5}$         \\ \hline
2D Heat           & 3(input)-128-128-1(output)            &            Sin, Sigmoid          &      $10^{-5}$         \\ \hline
2D Wave           & 3(input)-128-256-128-1(output)                       &        Sin, Tanh              &     $10^{-5}$             \\ \hline
Lotka-Volterra    & 1(input)-64-64-2(output)                       &         Sin             &    $10^{-4}$           \\ \hline
\end{tabular}
\label{nn_architecture}
\end{table}

\subsection{1D Transport equation} 1D transport equation consists of
\begin{align}
\label{1D_transport}
&\partial_{t}u+a\partial_{x}u = 0,\\
&u(0,x) =
  \begin{cases}
    \sin^4{0.25*\pi(x-0.1)}    & \quad \text{if } 0.1\leq x\leq0.5\\
    0 & \quad \text{otherwise}
  \end{cases}\nonumber,
\end{align}
where $a=\pi/10 (\simeq 0.314...)$. We have generated the observations from the analytic solution, by method of characteristics, of (\ref{1D_transport}).

\begin{center}
\begin{figure}[H]
    \centering
    \includegraphics[height=0.5\textwidth, width=\textwidth]{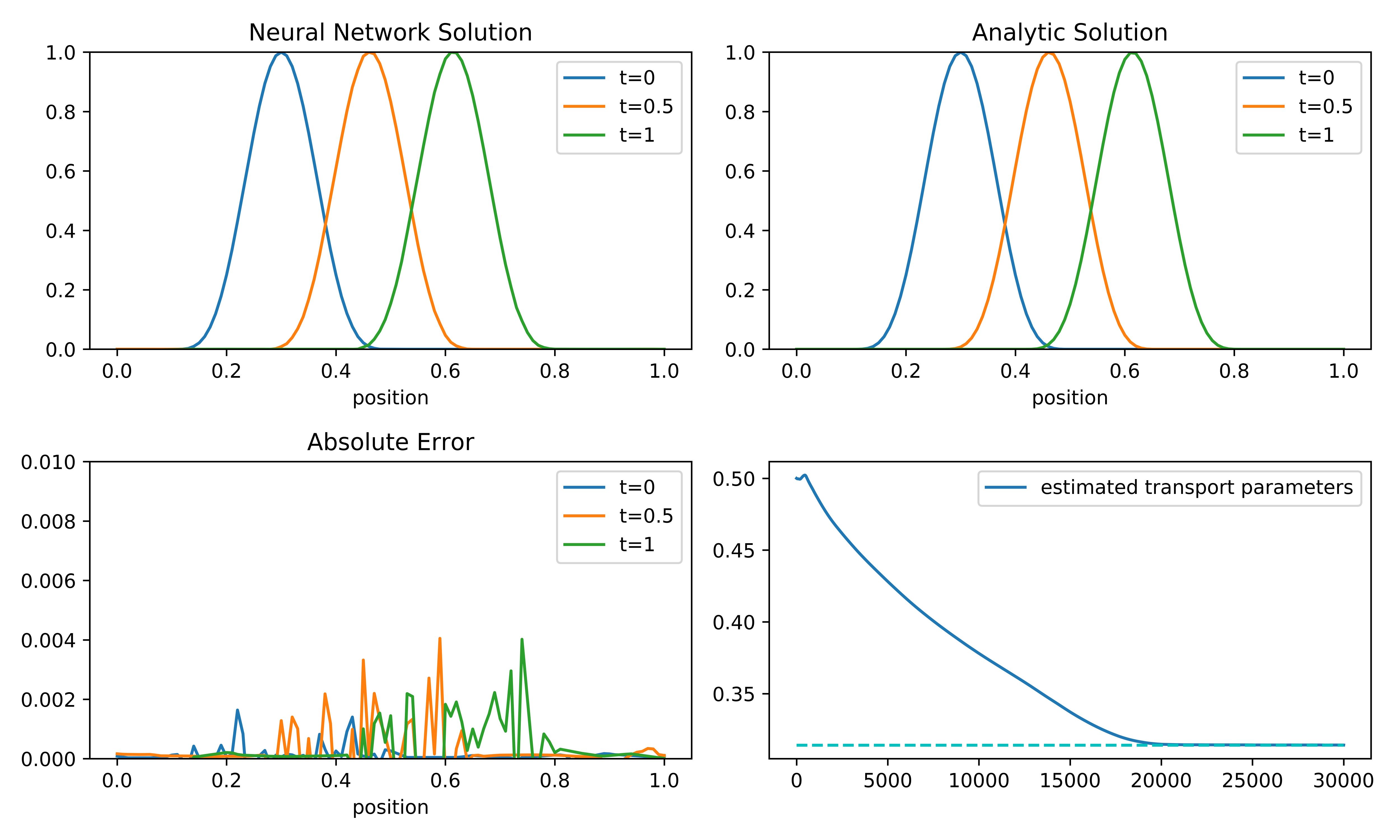}
    \caption{Experimental result for 1D transport equation}
    \label{fig:transport}
\end{figure}
\end{center}

\subsection{2D Heat equation}
\begin{align}
\label{2D_heat}
&\partial_{t}u = a^2 \left(\partial_{xx}u + \partial_{yy}u\right),\\ \nonumber
&u(t,0,y) = u(t,1,y) = 0,\\ \nonumber
&u(t,x,0) = u(t,x,1) = 0,
\end{align}
where $a=1$. We have generated the observations from the partial sum of the series solution, by separation of variables, of (\ref{2D_heat})

\begin{figure}[H]
    \centering
    \includegraphics[ height=0.5\textwidth, width=\textwidth]{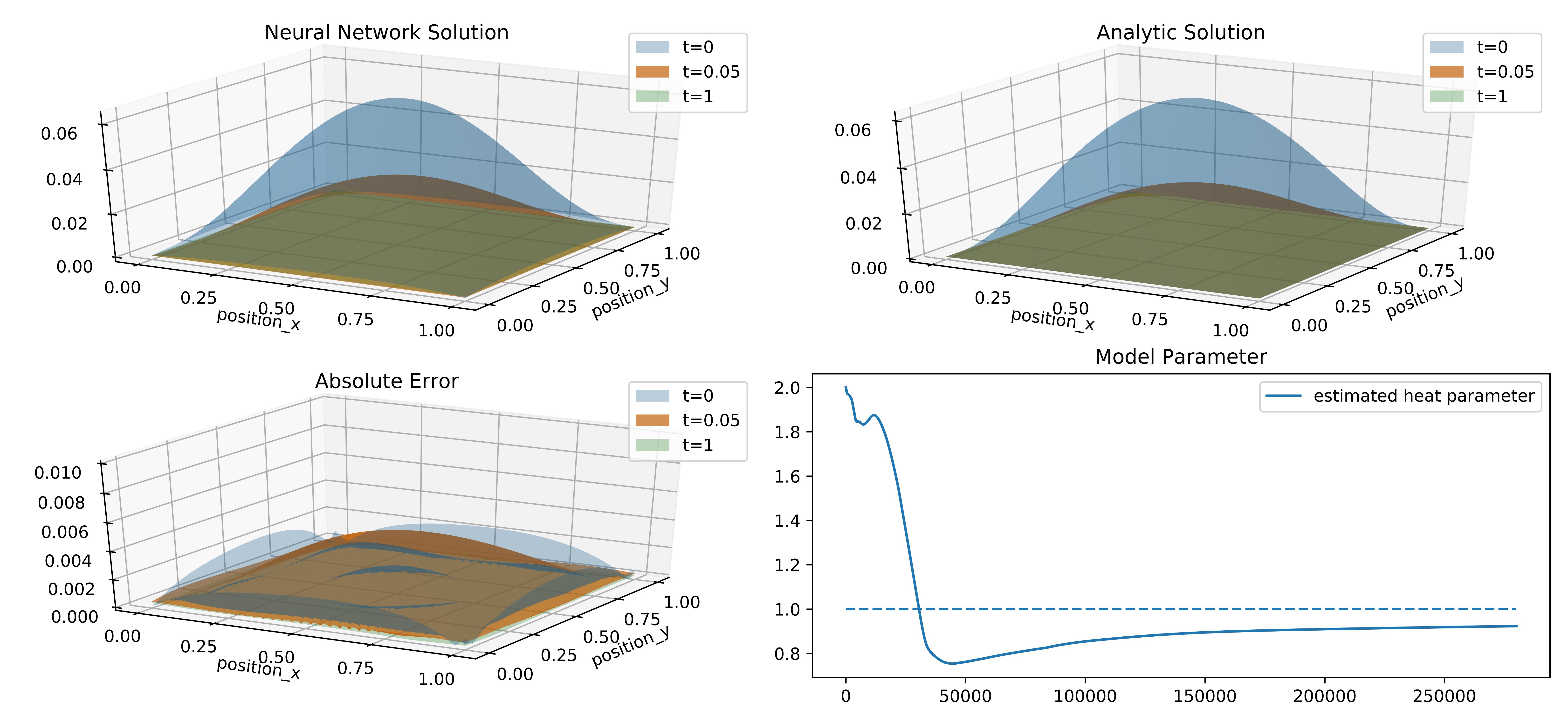}
    \caption{Experimental result for 2D heat equation}
    \label{fig:heat}
\end{figure}

\begin{figure}[H]
    \centering
    \includegraphics[ height=0.5\textwidth, width=\textwidth]{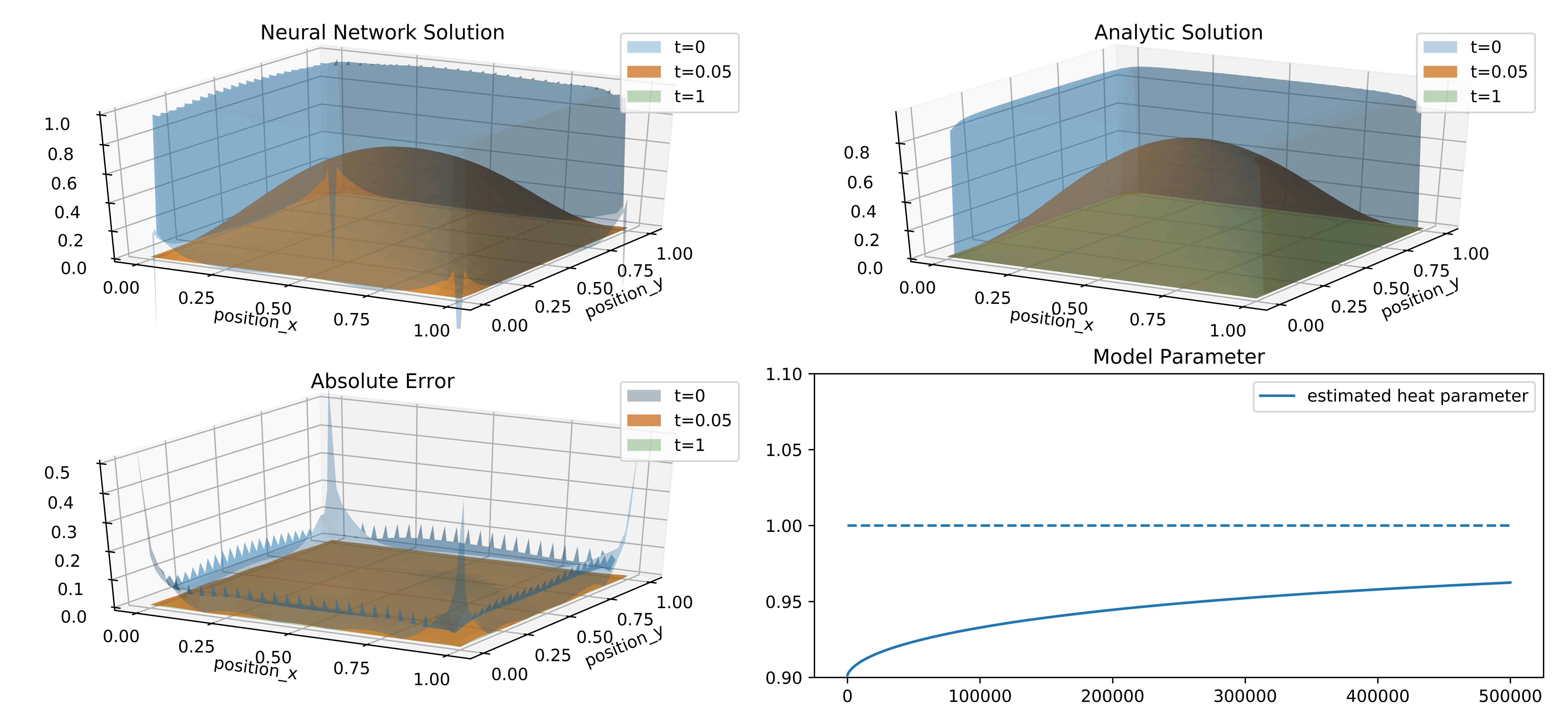}
    \caption{Experimental result for 2D heat equation}
    \label{fig:heat}
\end{figure}

\subsection{2D Wave equation}
\begin{align}
\label{2D_wave}
&\partial_{tt}u = a^2 \left(\partial_{xx}u + \partial_{yy}u\right),\\ \nonumber
&u(0,x) = xy(1-x)(1-y),\\ \nonumber
&\partial_{t}u(0,x,y) = 0,\\ \nonumber
&u(t,0,y) = u(t,1,y) = 0,\\ \nonumber
&u(t,x,0) = u(t,x,1) = 0,
\end{align}
where $a=1$. We have generated the observations from the partial sum of the series solution, by separation of variables, of (\ref{2D_wave})

\begin{figure}[H]
    \centering
    \includegraphics[ height=0.5\textwidth, width=\textwidth]{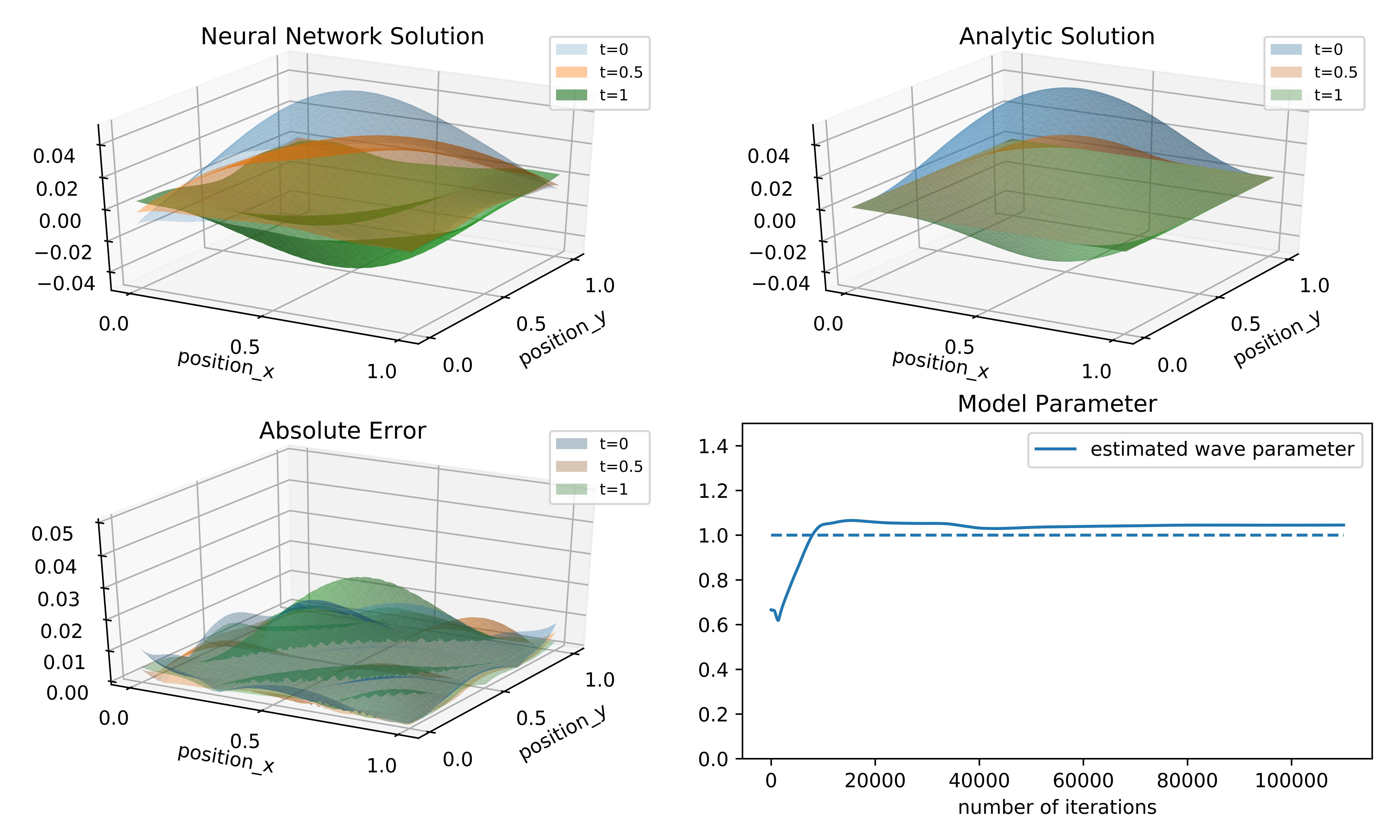}
    \caption{Experimental result for 2D wave equation}
    \label{fig:wave}
\end{figure}

\subsection{Lotka-Volterra system}

\begin{align}
\label{Lotka-Volterra}
&u'(t) = \alpha u -\beta uv,\\ \nonumber
&v'(t) = \delta uv - \gamma v,\\ \nonumber
&u(0) = 1, v(0) = 1,\\ \nonumber
\end{align}
where $\alpha=1, \beta=0.4, \delta=0.4, \gamma=0.1$. We have generated the observations from a numerical solution by the Runge-Kutta method of (\ref{Lotka-Volterra}). We used the $sin$ function as the activation function for Lotka-Volterra system. Considering the periodic nature of the solution, the periodic activation function is a natural choice.  

\begin{figure}[H]
    \centering
    \includegraphics[height=0.5\textwidth, width=\textwidth]{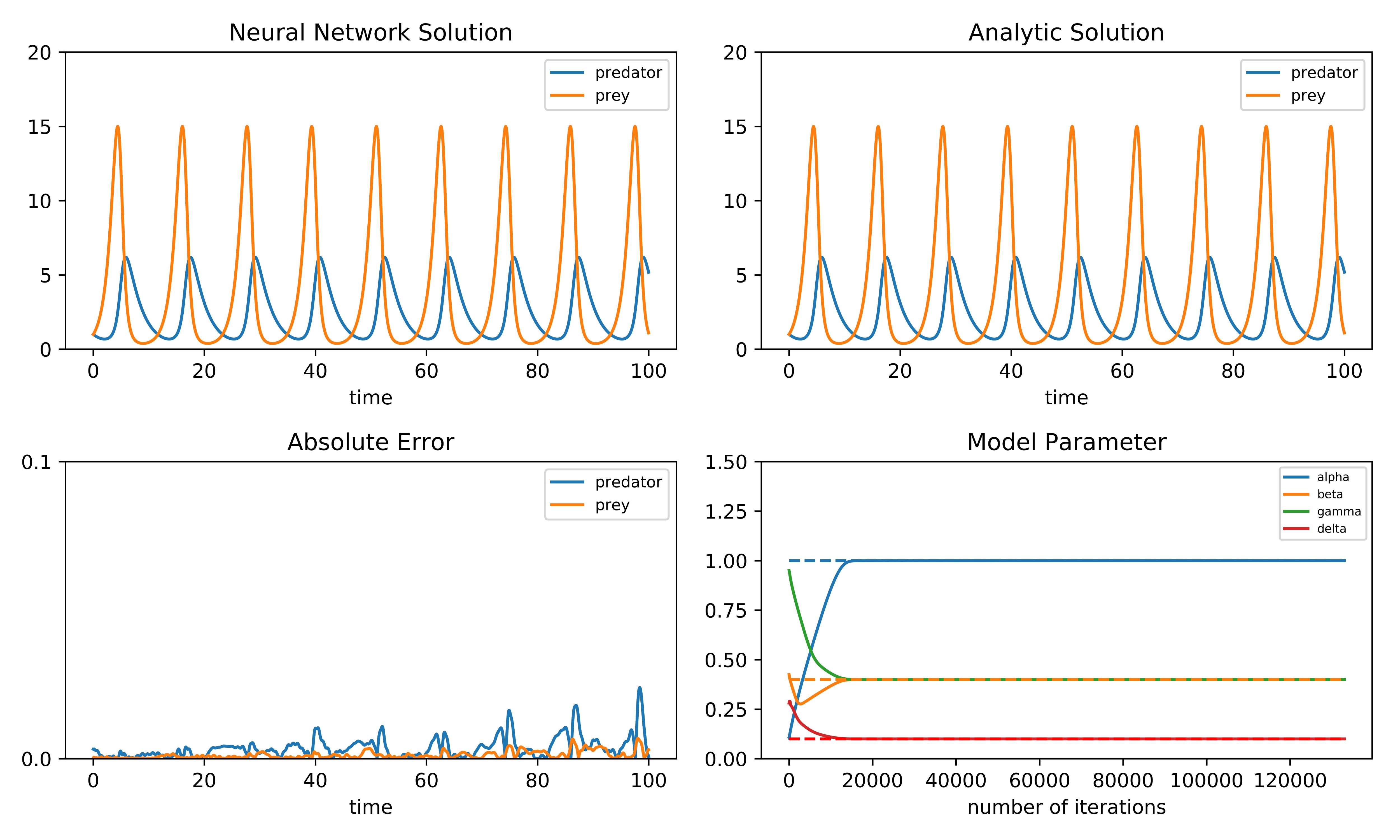}
    \caption{Experimental result for Lotka-Volterra equation}
    \label{fig:lotka_volterra}
\end{figure}

\subsection{Stability Condition}

In this section, we address the Courant$\mbox{-}$Friedrichs$\mbox{-}$Lewy (CFL) condition \cite{courant1967partial} which is a necessary condition while solving certain partial differential equations numerically. We compare the results of transport equation with three different Courant numbers which all violate the convergence condition. As we can see in figure \ref{fig:cfl}, our method shows the convergence well regradless of CFL condition.

\begin{figure}[H]
    \centering
    \includegraphics[ height=0.2\textwidth, width=\textwidth]{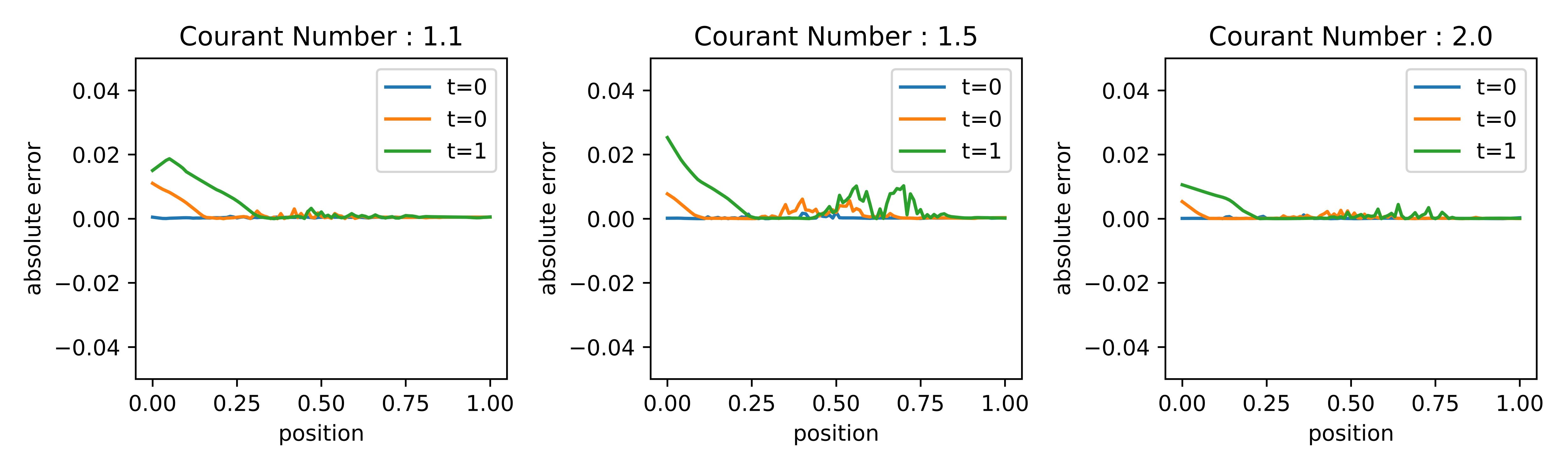}
    \caption{Experimental result for 2D wave equation}
    \label{fig:cfl}
\end{figure}

Note that our experiments in section 4.2 give nice results while the settings violate the well known stability condition called the Von Neumann stability condition.

\section{Conclusion} 
\label{sec5}
First we summarize our theoretical results. For linear differential equations, we have shown that the DNN solution can reduce the proposed loss function as much as we want. A key point in the proof is a continuation of the Theorem 2.1 in \cite{li1996simultaneous} which states the fact that a linear combination of the dilated activations can approximate the target function and the derivative of such a linear combination can approximate its derivative in $L_\infty$ sense. Next we have proved that the DNN which minimizes the loss converges to an analytic solution for linear parabolic or hyperbolic equations. In this step, we have applied basic energy estimates for each equation. Theoretical results for the inverse problem is also included as a continuation of the forward problem. We provide numerical experiments which show that our method indeed works accurately. We emphasize that our method can easily be implemented without any background knowledge about numerical analysis (for example, stability conditions) but about some libraries for implementing neural networks. Although it performs well for fundamental DEs, it might be hard to apply it to more complex equations. We have recognized that the error between a NN solution and an analytic/numerical solution is slightly increasing depending on time.  

For future directions, we may consider two problems. First, we can use more complicated neural network architectures such as CNN, RNN. Since we have dealt with time dependent PDEs, the combination of CNN and RNN would be a great choice for modelling. Second, the theoretical results for non-linear PDEs should be explored. The convergence results of our work are only applicable to linear PDEs. However, as in the experiment for the Lotka-Volterra system, our method is successful in approximating solutions even for non-linear systems. We hope proper convergence results for non-linear systems to be explored.

\section{Acknowledgement} 
This work was supported by the Basic Science Research Program through the National Research Foundation of Korea (NRF-2017R1E1A1A03070105, NRF-2019R1A5A1028324)

\section*{References}
\bibliography{bibfile}

\end{document}